\documentclass[oneside,10pt]{article}          
\usepackage[b5paper]{geometry}	    
\usepackage{amsfonts,amsmath,latexsym,amssymb} 
\usepackage{amsthm}                
\usepackage{mathrsfs,upref}         
\usepackage{mathptmx}		    
\usepackage{mia}	            

\usepackage[active]{srcltx}

\usepackage{graphicx}

\usepackage{enumerate}

\usepackage[final]
{showkeys}
\usepackage[hyphens]{url}

\RequirePackage[OT1]{fontenc}

\usepackage{tikz}
\usetikzlibrary{decorations}
\usetikzlibrary{decorations.pathreplacing}

\usepackage{hyperref}

\newtheorem{theorem}{Theorem}           
\newtheorem{lemma}{Lemma}

\theoremstyle{definition}

\newtheorem{remark}{Remark}

\newcommand{\card}{\operatorname{card}}

\newcommand{\ffrown}{\text{\raisebox{3pt}[0pt][0pt]{$\frown$}}}
\renewcommand{\O}{\underset{\ffrown}{<}}

\renewcommand{\gg}{>\kern-2pt>}
\renewcommand{\ll}{<\kern-2pt<}

\renewcommand{\S}{\operatorname{\mathsf{S}\!}}
\renewcommand{\S}{\mathsf{S}}

\renewcommand{\S}{\mathcal{S}}

\renewcommand{\gg}{>\kern-2pt>}
\renewcommand{\ll}{<\kern-2pt<}

\renewcommand{\le}{\leqslant}
\renewcommand{\ge}{\geqslant}

\newcommand{\Si}{\Sigma}

\renewcommand{\L}{\mathcal{S}}

\newcommand{\LD}{\mathcal{L}\!\mathcal{D}}
\renewcommand{\LD}{\mathcal{L}{\kern -1.9pt}\mathcal{D}}
\renewcommand{\LD}{\mathcal{D}}
\renewcommand{\LD}{\mathcal{L}{\kern -4pt}\mathcal{C}}
\renewcommand{\LD}{\mathcal{R}{\kern -3pt}\mathcal{C}}

\newcommand{\PP}{\operatorname{\mathsf{P}}}

\renewcommand{\PP}[1]{\mathcal{P}^{#1}_+}
\renewcommand{\PP}{\mathcal{P}}

\begin{document}

\title[Nearest-neighbor distances between points in a rectangle]{Exact upper bound on the sum of squared nearest-neighbor distances between points in a rectangle}


\author{Iosif Pinelis}

\address{Iosif Pinelis, Department of Mathematical Sciences\\
Michigan Technological University\\
Hough\-ton, Michigan 49931, USA
\email{ipinelis@mtu.edu}}

\CorrespondingAuthor{Iosif Pinelis}


\date{07.10.2018
}                               



        
%
%
%
%

\begin{abstract}
An exact upper bound on the sum of squared nearest-neighbor distances between points in a rectangle is given. 
\end{abstract}

\maketitle



\section{Introduction and summary}\label{intro}

For any natural $n\ge2$ and any positive real numbers $a$ and $b$, let $P_1,\dots,P_n$ be distinct points in an $a\times b$ rectangle $R$. For each $i\in[n]:=\{1,\dots,n\}$, let 
\begin{equation}\label{eq:d_i}
	d_i:=\min\big\{P_iP_k\colon k\in[n]\setminus\{i\}\big\},
\end{equation}
where $PQ$ denotes the Euclidean distance between points $P$ and $Q$. So, $d_i$ is the distance from the point $P_i$ to its nearest neighbor among the points $P_1,\dots,P_{i-1},P_{i+1},\dots,P_n$.  

\begin{theorem}\label{th:}
\begin{equation}\label{eq:}
	\sum_1^n d_i^2\le2a^2+2b^2. 
\end{equation}
The upper bound $2a^2+2b^2$ on $\sum_1^n d_i^2$ is exact in the following sense: it is attained (i) when $n=2$ and the points $P_1,P_2$ are opposite vertices of the rectangle $R$ and (ii) when $n=4$, $a=b$, and the points $P_1,\dots,P_4$ are the vertices of the square $R$. 
\end{theorem}

One might note here that, without further restrictions, the exact lower bound on $\sum_1^n d_i^2$ is the trivial bound $0$, ``attained in the limit" when the distinct points $P_1,\dots,P_n$ are arbitrarily close to one another. 


The terms used in definition \eqref{eq:d_i} are similar to those in the definition of the cells 
\begin{equation}
	C_i:=\big\{P\colon PP_i\le\min\nolimits_{k\in[n]\setminus\{i\}}PP_k\big\}  
\end{equation}
of the Voronoi diagram/tessellation $\{C_1,\dots,C_n\}$ generated by the points $P_1,\dots,P_n$. 
Voronoi diagrams \cite{aurenhammer} have very broad applications, not only in mathematics, but also in sciences, engineering, and other fields. 

The notion of nearest neighbors is used in various other ways as well, in particular in statistics \cite{berrett-samworth}, computational geometry \cite{agarwal-etal}, information theory \cite{gao-etal}, computer science \cite{andoni-etal}, biology \cite{janssen-etal}, and elsewhere.  


\section{Proof of Theorem~1
}\label{proof}

If the minimum in the definition \eqref{eq:d_i} of the $d_i$'s were replaced by the maximum, then it would be very easy to give the exact upper bound on $\sum_1^n d_i^2$. Indeed, since the largest distance between two arbitrary points of the rectangle $R$ is $\sqrt{a^2+b^2}$, this exact upper bound on $\sum_1^n d_i^2$ would be $n(a^2+b^2)$, attained when all the points $P_i$ are at two opposite vertices of the rectangle $R$. 

However, despite the great simplicity of the actual statement of Theorem~\ref{th:}, its proof is not at all simple. The reason for this is that the summands $d_i^2$ in the sum in \eqref{eq:} are the minima, rather than maxima, of a possibly large number of convex functions of the points $P_1,\dots,P_n$. So, the sum $\sum_1^n d_i^2$ as a function of $P_1,\dots,P_n$ is in general non-convex, usually with a large number of non-smooth local maxima (some of which may be global), plus a number of false (or quasi-) non-smooth local ``maxima''. Two of such possible situations are illustrated in Fig.~\ref{fig:pic}, where only one of the many variables on which the sum $\sum_1^n d_i^2$ depends is allowed to actually vary. One can see that the dependence of $\sum_1^n d_i^2$ on just one of the variables may already be quite complicated and 
exhibit quite different patterns. Of course, the complexity and variety of patterns of dependence of the sum $\sum_1^n d_i^2$ on all the variables involved (that is, on the all the coordinates of all the points $P_1,\dots,P_n$) is much greater yet. 

\begin{figure}[h]
	\centering
		\includegraphics[width=1.00\textwidth]{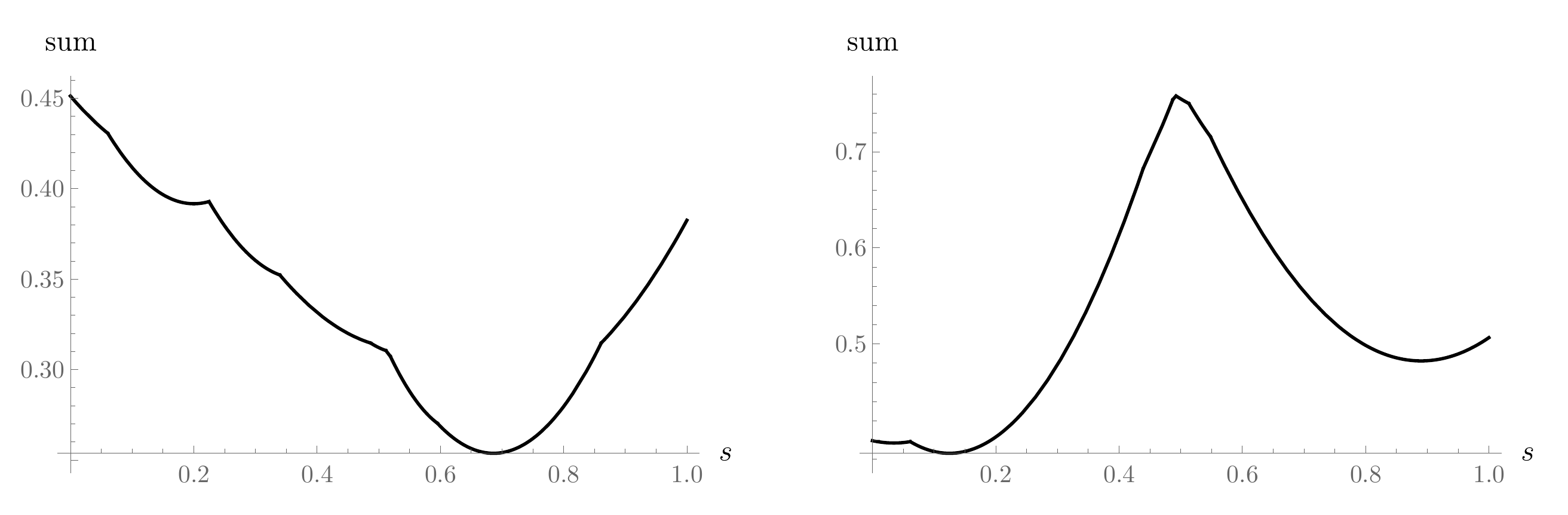}
	\caption{Each of the two graphs here is obtained by selecting $n=6$ pseudorandom points in the unit square $[0,1]^2$, replacing the abscissa of one of the $6$ points by a variable $s\in[0,1]$, and then plotting the resulting varying sum $\sum_1^6 d_i^2$ against $s$.}
	\label{fig:pic}
\end{figure}

To deal with the described difficulties, we will have to consider a variety of cases, in each of the cases carefully constructing convex functions that would be tight enough majorants, at least locally (in an appropriate sense), of the non-convex sum $\sum_1^n d_i^2$. Similar techniques could possibly be used elsewhere. 


\begin{remark}
One may realize at this point that, for each natural $n$, 
finding the exact upper bound on $\sum_1^n d_i^2$ is a problem of real algebraic geometry (also referred to as semialgebraic geometry); see e.g.\ \cite{basu-etal06}. In principle, for each given $n$ such a problem can be solved completely algorithmically. However, even for $n=3$, it takes minutes of computer time, in addition to quite a bit of preparation and manual post-processing of the computer algebra system output, to obtain the following exact upper bound on $\sum_1^n d_i^2=\sum_1^3 d_i^2$: 
\begin{equation}
S_3(a,b):=\left\{
\begin{aligned}
a^2+\frac{9 b^2}4 &\text{\quad if \quad} \frac ba\le\frac1{\sqrt3}, \\ 
\frac{3a^2}4+3 b^2 &\text{\quad if \quad} \frac1{\sqrt3}\le\frac ba\le\frac{\sqrt3}2, \\ 
12 \big(a^2-\sqrt3 a b+b^2\big) &\text{\quad if \quad} \frac{\sqrt3}{2}\le\frac ba\le\frac2{\sqrt3}, \\ 
3 a^2+\frac{3b^2}4 &\text{\quad if \quad} \frac2{\sqrt3}\le\frac ba\le\sqrt3, \\ 
\frac{9 a^2}4+b^2 &\text{\quad if \quad} \frac ba\ge\sqrt3.  
\end{aligned}
\right.
\end{equation}
It follows that 
\begin{equation}
	S_3(a,b)\le\frac{12}7\,(a^2+b^2), 
\end{equation}
with the equality if $b/a\in\{2/\sqrt3,\sqrt3/2\}$, and $\frac{12}7\approx1.71$. 
Also, $S_3(a,b)=\break 
6 (2 - \sqrt3)\,(a^2+b^2)\approx1.61\,(a^2+b^2)$ for $a=b$. 

For values of $n$ greater than $4$ (and also for $n=4$ with $a\ne b$), 
the problem of finding the exact upper bound on $\sum_1^n d_i^2$ is likely much more difficult than it is for $n=3$. \qed
\end{remark}


Let us now turn to the actual proof of Theorem~\ref{th:}. 
For any 
finite set $\S$ of points on 
the Euclidean plane, let $\Si(\S)$ denote the sum of squared nearest-neighbor distances between the points in $\S$. So, \eqref{eq:} can be rewritten as 
\begin{equation}\label{eq:rewr}
	\Si(\PP)\overset{\text{(?)}}\le2a^2+2b^2, 
\end{equation}
where 
\begin{equation*}
	\PP:=
	\{P_1,\dots,P_n\}. 
\end{equation*}

Let us prove \eqref{eq:rewr} by 
induction on $n$. For $n=2$, \eqref{eq:rewr} is obvious. Suppose now that $n\ge3$. 

Without loss of generality (wlog), $R=[0,a]\times[0,b]$. ``Partition" $R$ into the four congruent $\frac a2\times \frac b2$ rectangles 
\begin{equation*}
\text{$R_1:=[0,\tfrac a2]\times[0,\tfrac b2],\ R_2:=[\tfrac a2,a]\times[0,\tfrac b2],\ 
R_3:=[0,\tfrac a2]\times[\tfrac b2,b],\ R_4:=[\tfrac a2,a]\times[\tfrac b2,b]$.}	
\end{equation*}
For each $j\in\{1,\dots,4\}$, let 
\begin{equation*}
	\PP_j:=R_j\cap\PP\quad\text{and}\quad n_j:=\card\PP_j,
\end{equation*}
where $\card$ denotes the cardinality. 

\emph{Case 1: $n_1,\dots,n_4\ge2$}. Then $2\le n_j<n$ for all $j\in\{1,\dots,4\}$. So, by induction, 
\begin{equation*}
	\Si(\PP)\le\Si(\PP_1)+\cdots+\Si(\PP_4)\le4[2(\tfrac a2)^2+2(\tfrac b2)^2]= 2a^2+2b^2, 
\end{equation*}
which proves \eqref{eq:rewr} in Case 1.

The consideration of the next two cases is based in part on the following result. 

\begin{lemma}\label{lem:} Suppose that $n_1\ge2$ and $n_2=1$, so that there is exactly one point, say $Q$, in the set $\PP_2$. Then 
\begin{equation*}
	\Si(\PP_1)+d(Q,\PP_1)^2\le a^2+b^2.  
\end{equation*}
Here and in what follows, for any point $P$ and any set $\S$ of points, 
\begin{equation*}
	d(P,\S):=\inf_{A\in\S}PA, 
\end{equation*}
the shortest distance from $P$ to the points in $\S$. 
\end{lemma}


\begin{proof} Let $s$ be the largest abscissa of all the $n_1$ points in $\PP_1$, so that $0\le s\le\tfrac a2$. Let $(u,v):=Q$, so that $\tfrac a2\le u\le a$ and $0\le v\le\tfrac b2$. By vertical symmetry, wlog $0\le v\le\tfrac b4$, and hence $d(Q,\PP_1)^2\le(u-s)^2+(\tfrac b2-v)^2$. Also, by induction, $\Si(\PP_1)\le2s^2+2(\tfrac b2)^2$. Therefore and in view of the mentioned ranges of the variables $s,u,v$, 
\begin{align*}
	\Si(\PP_1)+d(Q,\PP_1)^2&\le2s^2+2(\tfrac b2)^2+(u-s)^2+(\tfrac b2-v)^2 \\ 
	&\le2s^2+2(\tfrac b2)^2+(a-s)^2+(\tfrac b2)^2 \\ 
	&\le a^2+\tfrac34\,b^2\le a^2+b^2,
\end{align*}
which proves Lemma 1. 
\end{proof}

\emph{Case 2: One of the $n_j$'s is $1$, and the other $n_j$'s are $\ge2$}. Here wlog 
$n_1,n_3,n_4\ge2$ and $n_2=1$. Then, in view of Lemma~\ref{lem:} and by induction, 
\begin{align*}
	\Si(\PP)&\le\Si(\PP_1)+d(Q,\PP_1)^2+\Si(\PP_3)+\Si(\PP_4) \\ 
	&\le a^2+b^2+2[2(\tfrac a2)^2+2(\tfrac b2)^2]= 2a^2+2b^2. 
\end{align*}

\emph{Case 3: Two of the $n_j$'s are $1$'s, and the other $n_j$'s are $\ge2$}. Here wlog either
$n_1,n_3\ge2$ and $n_2=n_4=1$ (the ``adjacent" subcase) or $n_1,n_4\ge2$ and $n_2=n_3=1$ (the ``non-adjacent" subcase). 
In the ``adjacent" subcase, let $Q$ be the only point in $\PP_2$, and let $T$ be the only point in $\PP_4$. Then, in view of Lemma~\ref{lem:} and by induction, 
\begin{equation*}
	\Si(\PP)\le\big(\Si(\PP_1)+d(Q,\PP_1)^2\big)+\big(\Si(\PP_3)+d(T,\PP_3)^2\big)\le(a^2+b^2)+(a^2+b^2)= 2a^2+2b^2. 
\end{equation*}
The ``non-adjacent" subcase is considered quite similarly, by interchanging $R_3$ and $R_4$. 

\emph{Case 4: Three of the $n_j$'s are $1$'s, and the other $n_j$ is $\ge2$}. Here wlog 
$n_1\ge2$ and $n_2=n_3=n_4=1$. Let $Q_2=(p,q),Q_3=(u,v),Q_4=(x,y)$ be the unique points in $\PP_2,\PP_3,\PP_4$, respectively, so that $p,x\in[\tfrac a2,a]$, $v,y\in[\tfrac b2,b]$, $u\in[0,\tfrac a2]$, $q\in[0,\tfrac b2]$. 

Let $s$ and $t$ be, respectively, the largest abscissa and the largest ordinate of all the $n_1$ points in $\PP_1$. Then  
\begin{equation}\label{eq:B4}
	\Si(\PP)\le B_4:=\Si(\PP_1)+Q_2S^2+Q_3T^2+\frac{Q_4Q_2^2+Q_4Q_3^2}2,
\end{equation}
where $S$ is a point (among the $n_1$ points in $\PP_1$) with the largest abscissa $s$, and $T$ is a point (among the $n_1$ points in $\PP_1$) with the largest ordinate $t$, so that $S=(s,\eta)$ for some $\eta\in[0,t]$, and $T=(\xi,t)$ for some $\xi\in[0,s]$. Of course, $Q_2S^2$ is a convex function of $\eta\in[0,t]$, and so, for $S_0:=(s,0)$ and $S_t:=(s,t)$ we have 
\begin{equation*}
	Q_2S^2\le\max(Q_2S_0^2,Q_2S_t^2)
	=(p - s)^2 + \max(q^2, (q-t)^2). 
\end{equation*}
Similarly, 
\begin{equation*}
	Q_3T^2\le
\max(u^2, (u-s)^2) + (v - t)^2. 
\end{equation*}
Next,
\begin{equation*}
	Q_4Q_2^2+Q_4Q_3^2=(x - p)^2 + (y - q)^2 + (x - u)^2 + (y - v)^2, 
\end{equation*}
which is convex in 
$(x,y)$, with the maximum in 
$(x,y)\in[\tfrac a2,a]\times[\tfrac b2,b]$ attained at 
$(x,y)=(a,b)$, for any given $p\in[\tfrac a2,a]$, $v\in[\tfrac b2,b]$, $u\in[0,\tfrac a2]$, $q\in[0,\tfrac b2]$. 
Further, by induction, $\Si(\PP_1)\le2s^2+2t^2$. Thus, by \eqref{eq:B4}, 
\begin{multline*}
	\Si(\PP)\le B_4\le \tilde B_4:=
	2 s^2+2t^2+\max \left(q^2,(q-t)^2\right)+\max \left((u-s)^2,u^2\right) \\ 
	+(p-s)^2+(v-t)^2 \\ 
	+\tfrac{1}{2} \left((a-p)^2+(b-q)^2\right)+\tfrac{1}{2} \left((a-u)^2+(b-v)^2\right). 
\end{multline*}
Clearly, $\tilde B_4$ is convex in 
$(p,v)$, and one can see that the maximum of $\tilde B_4$ in $(p,v)\in[\tfrac a2,a]\times[\tfrac b2,b]$ is attained at $(p,v)=(a,b)$, 
for any given $s\in[0,\tfrac a2]$, $t\in[0,\tfrac b2]$, $u\in[0,\tfrac a2]$, $q\in[0,\tfrac b2]$. So,
\begin{equation*}
	\tilde B_4\le B_{41}+B_{42}+B_{43}, 
\end{equation*}
where 
\begin{align*}
	B_{41}&:=\max \left(q^2,(q-t)^2\right)- b q+q^2/2, \\ 
	B_{42}&:=\max \left(u^2,(u-s)^2\right)- a u+u^2/2, \\ 
	B_{43}&:=\tfrac32(a^2+b^2)-2 a s-2 b t+3(s^2+t^2). 
\end{align*}
Since $B_{41}$ is convex in $q$ and $B_{42}$ is convex in $u$, one can see that 
\begin{equation*}
	B_{41}\le t^2\quad\text{and}\quad B_{42}\le s^2
\end{equation*}
for $q,t\in[0,\tfrac b2]$ and $u,s\in[0,\tfrac a2]$. So, 
\begin{align*}
	\Si(\PP)\le\tilde B_4&\le B_{41}+B_{42}+B_{43} \\ 
	&\le\tfrac32(a^2+b^2) +4 (s-\tfrac a2) s + 4 (t-\tfrac b2) t
	\le\tfrac32(a^2+b^2)\le2 a^2 + 2 b^2.  
\end{align*}


\emph{Case 5: This case obtains from Cases 2, 3, 4 by replacing there some of the conditions of the form $n_j=1$ by $n_j=0$}. This case immediately follows from Cases 2, 3, 4, because now the nonnegative contributions of the singleton sets $\PP_j$ corresponding to $n_j=1$ will be replaced by $0$ --- with the only exception occurring when three of the $n_j$'s are $0$ and hence the remaining one of the $n_j$'s (say $n_1$) equals $n$. Indeed, in the latter exceptional subcase, we cannot use the induction, since $n_1=n\not<n$. The remedy in this case is to continue the ``partitioning" of the smaller rectangles containing all the $n$ points into yet smaller congruent rectangles until we no longer have such an exceptional situation. This process will stop. Indeed, if it never stopped, then all the $n$ distinct points (with $n\ge3$) would be eventually contained in a singleton set, which is a contradiction.

\medskip
\hrule
\medskip

So far, we have considered all the cases when at least one of the $n_j$'s is $\ge2$. Otherwise, we have $n_j\le1$ for all $j\in\{1,\dots,4\}$ and hence $n\le n_1+\cdots+n_4\le4$. Since $n\ge3$, it remains to consider the following two cases.     

\emph{Case 6: $n=3$}. By shrinking the rectangle $R$ horizontally and vertically, we can obtain a possibly smaller rectangle $\tilde R\subseteq R$, with side lengths $\tilde a\le a$ and $\tilde b\le b$, such that $\tilde R$ still contains all the points $P_1,P_2,P_3$ and each side of $\tilde R$ contains at least one of the three points $P_1,P_2,P_3$. If we can then show that $\Si(\{P_1,P_2,P_3\})\le2\tilde a^2+2\tilde b^2$, then the desired inequality $\Si(\{P_1,P_2,P_3\})\le2a^2+2b^2$ will obviously follow. So, 
wlog we may assume that each side of $R$ contains at least one of the three points $P_1,P_2,P_3$. Then, by the pigeonhole principle, at least two sides of the rectangle $R$ must share one of the three points. Also, wlog none of the points $P_1,P_2,P_3$ is in the interior of $R$. Indeed, if e.g. $P_3$ is in the interior of $R$, then we can move $P_3$ away from the line $\ell(P_1,P_2)$ (through $P_1,P_2$) in the direction perpendicular to $\ell(P_1,P_2)$ (till we hit the boundary of $R$), so that all the pairwise distances between the points $P_1,P_2,P_3$ may only increase, and then $\Si(\PP)$ may only increase.  

Hence, wlog $P_1=(0,0),P_2=(u,b),P_3=(a,v)$ for some $u\in[0,a]$ and $v\in[0,b]$. So, 
\begin{equation*}
	\Si(\PP)\le B_6:=\frac{P_1P_2^2+P_1P_3^2}2+\frac{P_2P_1^2+P_2P_3^2}2+\frac{P_3P_2^2+P_3P_1^2}2=
	P_1P_2^2+P_1P_3^2+P_2P_3^2. 
\end{equation*}
Since $B_6$ is convex in 
$(u,v)$, its maximum in $(u,v)\in[0,a]\times[0,b]$ is attained when $(u,v)\in\{0,a\}\times\{0,b\}$, and this maximum is $2a^2+2b^2$. 
More explicitly, one may also note that 
\begin{equation*}
	\tfrac12\,B_6-(a^2+b^2)= - (a - u) u - (b - v) v\le0. 
\end{equation*}
This proves Case 6. 

\emph{Case 7: $n=4$}. Again, by shrinking the rectangle $R$, wlog we may assume that each side of $R$ contains at least one of the four points $P_1,P_2,P_3,P_4$. Also, wlog each of the four points $P_1,P_2,P_3,P_4$ is either (i) in the convex hull of the other three points or (ii) on the boundary of $R$. Indeed, otherwise wlog $P_4$ is in the interior of $R$, but not in the convex hull of $P_1,P_2,P_3$. Then there is a closed half-plane, say $H$, containing the points $P_1,P_2,P_3$ but not containing the point $P_4$. So,   
%
%
%
%
then we can move $P_4$ away from the 
half-plane $H$ in the direction perpendicular to 
the boundary line of $H$ (till we hit the boundary of $R$), so that all the pairwise distances between the points $P_1,P_2,P_3,P_4$ may only increase, and then $\Si(\PP)$ may only increase. Therefore, wlog we have one of the following two subcases. 

\emph{Subcase 7.1: $P_4$ is in the convex hull of $P_1,P_2,P_3$.} Then $P_4=(1-s-t)P_1+sP_2+tP_3$ for some $s,t$ such that $s,t\ge0$ and $s+t\le1$. 
Also, here, as in Case 6, wlog $P_1=(0,0),P_2=(u,b),P_3=(a,v)$ for some $u\in[0,a]$ and $v\in[0,b]$. Hence
\begin{equation*}
	\Si(\PP)\le B_{71}:=P_1P_4^2+P_2P_4^2+P_3P_4^2+[(1-s-t)P_4P_1^2+s\,P_4P_2^2+t\,P_4P_3^2]. 
\end{equation*}
Clearly, $B_{71}$ is convex in $(u,v)$ and hence attains its maximum in $(u,v)\in[0,a]\times[0,b]$ at one of the four vertices of the rectangle $[0,a]\times[0,b]$. To complete the consideration of Subcase 7.1, it 
remains to note that, given the above conditions on $s$ and $t$,  
\begin{align*}
	B_{71}\big|_{u=0,\,v=0}-(2a^2+2b^2)&=a^2 (t-1) (2 t+1)+b^2 (s-1) (2 s+1)\le0, \\ 
	B_{71}\big|_{u=a,\,v=0}-(2a^2+2b^2)&=a^2 (s+t) (2 (s+t)-3)+b^2 (s-1) (2 s+1)\le0, \\ 
	B_{71}\big|_{u=0,\,v=b}-(2a^2+2b^2)&=a^2 (t-1) (2 t+1)+b^2 (s+t) (2 (s+t)-3)\le0, \\ 
	B_{71}\big|_{u=a,\,v=b}-(2a^2+2b^2)&=\left(a^2+b^2\right) (s+t) (2 (s+t)-3)\le0.  
\end{align*}

Alternatively, one may note that the Hessian matrix of $B_{71}$ with respect to $s$ and $t$ equals $4G$, where $G$ is the Gram matrix of the vectors $P_2[=\overrightarrow{P_1P_2}=(u,b)]$ and $P_3[=\overrightarrow{P_1P_3}=(a,v)]$. So, $B_{71}$ is convex in $(s,t)$, and hence attains its maximum in $(s,t)$ at one of the points $(0,0),(0,1),(1,0)$. To complete the consideration of Subcase 7.1 this other way, it remains to note that 
%
%
\begin{align*}
B_{71}\big|_{s=0,\,t=0}-(2 a^2 + 2 b^2)&=(u^2-a^2)+(v^2-b^2)\le0, \\ 
B_{71}\big|_{s=0,\,t=1}-(2 a^2 + 2 b^2)&=
-(2a - u)u - 2 (b-v)v - b^2\le0, \\ 
B_{71}\big|_{s=1,\,t=0}-(2 a^2 + 2 b^2)&=
-(2b - v)v- 2 (a-u)u - a^2\le0. 
\end{align*}



\emph{Subcase 7.2: All the four points $P_1,P_2,P_3,P_4$ are on the boundary of $R$.} 

\emph{Subsubcase 7.2.1: None of the points $P_1,P_2,P_3,P_4$ is shared by any two sides of the rectangle $R$
.} 
So, wlog $P_1=(s,0),P_2=(0,t),P_3=(u,b),P_4=(a,v)$ for some $s,u\in[0,a]$ and $t,v\in[0,b]$. Then, similarly to the case of $n=3$ (that is, Case~6), here   
\begin{align*}
	\Si(\PP)\le B_{721}&:=\frac{P_1P_2^2+P_1P_4^2}2+\frac{P_2P_1^2+P_2P_3^2}2
	+\frac{P_3P_2^2+P_3P_4^2}2+\frac{P_4P_1^2+P_4P_3^2}2 \\ 
	&=P_1P_2^2+P_2P_3^2+P_3P_4^2+P_4P_1^2 \le2a^2+2b^2, 
\end{align*}
by convexity. More explicitly, one may also note that 
\begin{equation*}
	\tfrac12\,B_{721}-(a^2+b^2)=-(a - s) s - (b - t) t - (a - u) u - (b - v) v\le0. 
\end{equation*}
Subsubcase 7.2.1 is done. 

\emph{Subsubcase 7.2.2: One of the points $P_1,P_2,P_3,P_4$ (say $P_1$) is shared by two sides of the rectangle $R$.} So, wlog $P_1=(0,0)$. Suppose that one of the two sides of $R$ (say $S_1$ and $S_2$) sharing the point $P_1$ contains one of the points $P_2,P_3,P_4$; let us say this side is $S_1$. Then we can move $P_1$ slightly along the side $S_2$ out of its position at $(0,0)$. 
In view of the continuity of $\Si(\PP)$ in $P_1$, we can thereby get rid of the sharing and thus reduce Subsubcase 7.2.2 to Subsubcase 7.2.1 -- provided that at least one of the two sides of $R$ sharing the point $P_1$ contains one of the points $P_2,P_3,P_4$. So, wlog $P_1=(0,0)$ and none of the two sides of $R$ sharing the point $P_1$ contains any of the points $P_2,P_3,P_4$. So, one of the sides of $R$ not sharing the point $P_1$ contains two of the points $P_2,P_3,P_4$. Therefore and by the interchangeability of the horizontal and vertical directions, wlog $P_1=(0,0),P_2=(u,b),P_3=(a,v),P_4=(a,w)$ for some $u\in[0,a]$ and $v,w\in[0,b]$ such that $v>w$. So,   
\begin{equation*}
	\Si(\PP)\le B_{722}:=P_1P_2^2+P_2P_3^2+P_3P_4^2+P_4P_1^2\le2a^2+2b^2, 
\end{equation*}
again by convexity. 
More explicitly, one may also note that 
\begin{equation*}
	\tfrac12\,B_{722}-(a^2+b^2)=-(a - u) u - (b - v) v - (v - w) w\le0. 
\end{equation*}
Subsubcase 7.2.2 is done, as well as the entire proof of 
Theorem~\ref{th:}.  \qed

\bigskip

The special case of Theorem~\ref{th:} with $a=b=1$ was conjectured by T.\ Amdeberhan~\cite{pts-in-sqr}. 


\bibliographystyle{abbrv}

\bibliography{P:/pCloudSync/mtu_pCloud_02-02-17/bib_files/citations10.13.18a}

%
%
%
%
%
%
%
%
%
%

\end{document}